\documentclass[10pt]{amsart}
\usepackage{a4wide}
\usepackage{enumerate}
\usepackage{amssymb}
\usepackage{amsthm}
\usepackage{exscale}
\usepackage{hyperref}
\newcommand{\RR}{\mathbb{R}}

\newcommand{\NN}{\mathbb{N}}

\newcommand{\EE}{\, \mathbb{E} \,}
\newcommand{\PP}{\mathbb{P}}
\newcommand{\cA}{\mathcal{A}}
\newcommand{\cC}{\mathcal{C}}

\newcommand{\Tr}{{\mathop{\mathrm{Tr} \,}}}

\renewcommand{\L}{\Lambda}

\newcommand{\loc}{{\mathop{\mathrm{loc}}}}

\newcommand{\be}{\begin{equation}}
\newcommand{\ee}{\end{equation}}

\newtheorem{thm}{Theorem}
\newtheorem{lem}[thm]{Lemma}

\newtheorem{cor}[thm]{Corollary}
\theoremstyle{definition}

\theoremstyle{remark}

\newcommand{\hm}[1]{\leavevmode{\marginpar{\tiny%
$\hbox to 0mm{\hspace*{-0.5mm}$\leftarrow$\hss}%
\vcenter{\vrule depth 0.1mm height 0.1mm width \the\marginparwidth}%
\hbox to
0mm{\hss$\rightarrow$\hspace*{-0.5mm}}$\\\relax\raggedright #1}}}

\begin{document}
\setlength{\parindent}{0em}

\title[Continuity of the integrated density of states for Gaussian potentials]
{Lipschitz-continuity of the integrated density of states for Gaussian random potentials}

\author[I.~Veseli\'c] {Ivan Veseli\'c}
\address{Fakult\"at f\"ur Mathematik,\, 09107\, TU-Chemnitz, Germany  }
\urladdr{www.tu-chemnitz.de/mathematik/stochastik}

%

 \thanks{To appear in \href{http://www.springerlink.com/content/0377-9017}{Letters in Mathematical Physics}
with 
\href{http://dx.doi.org/10.1007/s11005-011-0465-1}{DOI: 10.1007/s11005-011-0465-1}
}

\begin{abstract}
The integrated density of states of a Schr\"odinger operator with random potential given by a homogeneous Gaussian field 
whose covariance function is continuous, compactly supported and has positive mean, is locally uniformly Lipschitz-continuous. 
This is proven using a Wegner estimate.
\end{abstract}

\maketitle

Let $d \in \NN$, $\|x\|: = (x_1^2+ \dots +x_d^2)^{1/2}$ and $|x| := \max (|x_1|, \dots , |x_d|)$ for  any $x \in \RR^d$,
$\L_R := \{x \in \RR^d \mid |x| <R\}$, $A\colon \RR^d\to \RR^d$ be measurable with $x\mapsto |A(x)|$ in $L^2_{\loc}(\RR^d)$,
$(\Omega,\cA, \PP)$  a complete probability space  with associated expectation $\EE$, $V\colon \Omega \times  \RR^d\to \RR$ 
a separable, jointly measurable version of a $\RR^d$-homogeneous Gaussian stochastic field with zero mean and covariance function $x\mapsto \cC(x):=\EE(V(\cdot, x)V(\cdot, 0))$
which is continuous at $x=0$ and satisfies $\cC(0) \in (0,\infty)$. .
For each $L\in (0,\infty)$ the restricted random operator $H_{\omega,L} := \sum_{j=1}^d (i \frac{\partial}{\partial x_j} +A_j)^2+ V(\omega, \cdot)$ on $\L_L$
with either Neumann or Dirichlet boundary conditions is almost surely selfadjoint, lower semibounded and has purely discrete spectrum. For a detailed dicussion of these facts see
\cite{HupferLMW-01a}. The hypotheses stated so far will be referred to as (H). The indicator function of a set $S$ is denoted by $\chi_S$.

\begin{thm}
Assume that (H) holds as well as
\begin{equation}
\cC \text{ is supported in  $\L_R$ and } \bar\cC := \int dx\, \cC(x)  >0.
\setcounter{equation}{2}
\end{equation}
Then there exists a isotone function $C_{WG}:\RR\to\RR$ depending only the covariance $\cC$
such that for all $L\in [1, \infty)$, $E_1\leq E_2\in\RR$ and both choices of b.{}c.{} in (H) the Wegner estimate
\begin{equation}
N_L(E_2)-N_L(E_1)  \leq C_{WG}(E_2) \, (2L)^d \, (E_2-E_1)
\end{equation}
holds, where $N_L(E) := \EE \big \{  \Tr \chi_{(-\infty,E]} (H_{\omega,L})\big\} $.
\end{thm}

In \cite[Thm.~1]{FischerHLM-97b} and and \cite[Cor.~4.3]{HupferLMW-01a}  the same statement as in the above Theorem  is derived, in the case that assumption (2)
is replaced by
\begin{itemize}
 \item[(4)] There exists a finite signed Borel-measure $\mu$ on $\RR^d$, an open $\Gamma \subset \RR^d$ and $\gamma \in  (0,\infty)$ such that
 \begin{displaymath}
 \int \mu(dx)\int \mu(dy) \cC(x-y) = \cC(0) \text{ and } \mu \ast \cC \ge \cC(0) \, \gamma \, \chi_\Gamma.
 \end{displaymath}
\end{itemize}
In the references the reader may find explicit bounds on the function $C_{WG}$. 
To prove Theorem 1 it is sufficient to show that (2) implies (4).
Before we do so in Lemma~6 below, let us state specific cases under which it was known before (cf.~the references in \cite{HupferLMW-01a})
that a measure $\mu$ as in (4) exists:
\setcounter{equation}{4}
\begin{subequations}
\begin{equation}
\cC(x) \ge 0 \text{ for all $x\in \RR^d$ and $\cC$ not identically vanishing, }
\end{equation}
\begin{equation}
d=1, \quad  \cC(x) = \int w(x-y)w(y) dy, \text{and } w=\chi_{[-3,3]}-\frac{5}{4}\chi_{[-1,1]},
\end{equation}
\begin{equation}
\cC(x) = \cC(0) \exp\big(-\|x\|^2/(2 t^2)\big)  \big( 1-7 \|x\|^2/(16 t^2) + \|x\|^4/(32 t^4)\big), t>0 \text{ arbitrary}.
\end{equation}
\end{subequations}
In all three listed cases the integral $\bar \cC$ is positive. 
In examples (5a) and (5c) the covariance function need not be of compact support.
Note that while (4) is an infinite family of conditions (one for each $x\in \RR^d$),
condition (2) is one-dimensional.

\setcounter{thm}{5}
\begin{lem}
Assume (H) and (2). Choose $b$ positive with $b\leq \bar\cC\, (2 \mathrm{e}R \|\cC\|_1)^{-1} $ and $f(x) = \mathrm{e}^{-b|x|}$.
Then for all $x \in \RR^d$ we have
\setcounter{equation}{6}
\begin{equation}
 (f\ast\cC)(x):= \int dy \, f(y) \, \cC(x-y) \ge \frac{\bar\cC}{2} f(x).
\end{equation}
In particular, condition (4) holds.
\end{lem}
\begin{proof}
 Since
\begin{equation*}
 (f\ast\cC)(x)= \bar \cC f(x)+ \int_{x+\L_R} dy \big(f(y)-f(x) \big) \cC(x-y),
\end{equation*}
(7) holds, if the absolute value of the second term is bounded by  $ f(x)\,\bar\cC/2$. Note that 
\begin{equation}
\Big| \int_{x+\L_R} dy \big(\mathrm{e}^{-b|y|}-\mathrm{e}^{-b|x|} \big) \cC(x-y)    \Big|
\le
\mathrm{e}^{-b|x|}  \int_{x+\L_R} dy \big(\mathrm{e}^{b|x|-b|y|}-1 \big) |\cC(x-y)|.
\end{equation}
For $|x-y|\leq R$ we have  $\big| b|x|-b|y|\big| \leq \frac{\bar\cC}{2 \mathrm{e} \|\cC\|_1}	 \leq \frac{1}{2 \mathrm{e} } $
and thus $\mathrm{e}^{b|x|-b|y|}-1   \leq \frac{\bar\cC}{2 \|\cC\|_1}$.
Hence
\begin{equation*}
(8) \leq 
\mathrm{e}^{-b|x|}  \frac{\bar\cC}{2 \|\cC\|_1}\int_{x+\L_R} dy  \,|\cC(x-y)|= \frac{\bar\cC}{2} f(x).
\end{equation*}
Now we show that the measure $\mu(dx) := \alpha f(x) dx$ with an appropriate choice of $\alpha\in \RR$
satisfies (4).  Ineq.~(7) implies $\int dx\int dy f(x)f(y) \cC(x-y)  \geq \frac{\alpha^2 \bar\cC}{2} \|f\|_2^2>0$. 
Thus the choice $\alpha:= \sqrt{\cC(0)} \big(\int dx\int dy f(x)f(y) \cC(x-y)\big)^{-1/2}$
is well defined and implies $ \int \mu(dx)\int \mu(dy) \cC(x-y) = \cC(0) $. If we set $\gamma =\frac{\alpha \bar\cC}{2\mathrm{e} \cC(0)}$
and $\Gamma =\Lambda_{1/b} \subset \RR^d$, then
\begin{equation*}
(\mu\ast\cC)(x):= \alpha \int dy \, f(y) \, \cC(x-y) \ge \alpha \frac{\bar\cC}{2} \, \mathrm{e}^{-b|x|}
\geq \cC(0)\, \gamma\, \chi_\Gamma(x).
\end{equation*}
Thus condition (4) is satisfied.
\end{proof}
Under appropriate conditions on $A\colon \RR^d\to  \RR^d$ (in particular for $A\equiv0$) it is known that
\begin{itemize}
 \item[(9)] 
there exists an isotone, right-continuous function $N\colon \RR\to \RR$ such that  $\lim\limits_{L\to \infty} (2L)^{-d} N_L(E)  =N(E)$ holds for every $E$ where $N$ is continuous.
%
%
\end{itemize}

\setcounter{thm}{9}
\begin{cor}
 Assume (H), (2), and (9). Then  we have for all $E_1\le E_2 \in\RR$
\setcounter{equation}{10}
\begin{equation}
N(E_2) -N(E_1) \leq C_{WG}(E_2)\, (E_2-E_1).
\end{equation}
Thus $N$ is locally uniformly Lipschitz-continuous, hence  differentiable a.e., and (11) implies an upper bound on its derivative, the \emph{density of states}. 
\end{cor}
For more background information on this short note see the references. The author thanks P.~M\"uller for enlightening discussions.

\def\cprime{$'$}\def\polhk#1{\setbox0=\hbox{#1}{\ooalign{\hidewidth
  \lower1.5ex\hbox{`}\hidewidth\crcr\unhbox0}}}

\end{document}